\documentclass[a4paper]{article}

\usepackage{amsmath, amsthm, amssymb, mathrsfs}
\usepackage{graphicx, tikz}
\usepackage{hyperref}

\newtheorem{theorem}{Theorem}[section]
\newtheorem{definition}{Definition}[section]
\newtheorem{problem}{Problem}[section]

\newtheorem{remark}{Remark}[section]

\begin{document}

\title{Uniqueness in inverse cavity scattering problems with phaseless near-field data}

\author{
Deyue Zhang\thanks{School of Mathematics, Jilin University, Changchun, China. {\it dyzhang@jlu.edu.cn}}, 
Yinglin Wang\thanks{School of Mathematics, Jilin University, Changchun, China. {\it wyl15022015166@163.com}}, 
Yukun Guo\thanks{School of Mathematics, Harbin Institute of Technology, Harbin, China. {\it ykguo@hit.edu.cn} (Corresponding author)}\ \ and 
Jingzhi Li\thanks{Department of Mathematics, Southern University of Science and Technology, Shenzhen, China. {\it li.jz@sustech.edu.cn}}
}
\date{}

\maketitle

\begin{abstract}
This paper is concerned with the uniqueness of inverse acoustic scattering problem for cavities with the modulus of the near-fields. With the aid of the reference ball technique and the superpositions of two point sources as the incident waves, we rigorously prove that the location and shape of the cavity as well as its boundary condition can be uniquely determined by the modulus of near-fields at an admissible surface. To our knowledge, this is the first uniqueness result in inverse cavity scattering problems with phaseless near-field data. In this paper, we make use of the phaseless near-field data incurred by the cavity and the point sources, and thus the configuration is more feasible in practice.
\end{abstract}



\section{Introduction}

The inverse obstacle scattering problems are of significant importance in diverse areas of sciences and technology such as non-destructive testing, radar sensing, sonar detection and biomedical imaging (see, e.g. \cite{Colton}), which are typical exterior inverse scattering problems. However, the interior inverse scattering problems for determining the shape of cavities arise in many practical applications of radar sensing and non-destructive testing (see, e.g. \cite{Jakubik, Qin1}). In contrast to the typical exterior inverse scattering problems, the interior inverse scattering problems are more complicated to some extent due to the repeated reflections of the scattered waves, and some mathematical
studies have been made. In \cite{Qin1, Zeng, Liu14}, the uniqueness of the inverse cavity scattering with the Dirichlet boundary condition has been established. For the the impedance boundary condition and the mixed boundary condition, the uniqueness results have been  given in \cite{Qin2, Liu14} and \cite{Hu}, respectively. In \cite{Qin}, the authors proposed the method of adding an artificial obstacle to avoid the interior eigenvalues and gave a new proof for the uniqueness of the inverse problems. There have also been some numerical reconstruction algorithms for solving the inverse cavity problems. We refer to \cite{Qin1, Qin2, Hu, Zeng1, Qin0, Zeng, Liu14, Sun} for the linear sampling method, the regularized Newton iterative method, the decomposition method, the factorization method and the reciprocity gap functional method.

The above theories and numerical methods are based on the full data (both the intensity and phase). However, in many situations, one can measure only the intensity/magnitude of the data, which leads to the study of inverse scattering problems with phaseless or intensity-only data.

The exterior inverse scattering problems with phaseless near-field data have been studied numerically (see, e.g. \cite{Candes, Candes1, Caorsi, CH17, CFH17, Maleki, Maleki1, Pan, Takenaka}), and few results have been done on the theory of uniqueness for the inverse scattering problems. A recent result on uniqueness in \cite{Kli14} was related to the reconstruction of a potential with the phaseless near-field data for point sources on a spherical surface and an interval of wave-numbers, which was extended in \cite{Kli17} for determining the wave speed in generalized 3-D Helmholtz equation. The uniqueness of a coefficient inverse scattering problem with phaseless near-field data has been established in \cite{KR17}. We also refer to \cite{KR16, Novikov15, Novikov16} for some recovery algorithms for the inverse medium scattering problems with phaseless near-field data. The stability analysis for linearized near-field phase retrieval in X-ray phase contrast imaging can be found in \cite{Maretzke}.

For exterior inverse scattering problems with phaseless far-field data, several uniqueness results have been established. With a priori information, uniqueness on determining the radius of a sound-soft ball was given in \cite{LZ09}. A method of superposition of incident waves was proposed in \cite{ZhangBo20171}, which led to the multi-frequency Newton iteration algorithm \cite{ZhangBo20171, ZhangBo20172} and the fast imaging algorithm \cite{ZZ18}. Moreover, uniqueness results were established in \cite{XZZ18a} under some a priori assumptions. Recently, the idea of resorting to the reference ball technique (see, e.g. \cite{Colton1, Colton2, Li}) in phaseless inverse scattering problems was  proposed by Zhang and Guo in \cite{ZhangDeyue20181}, and the uniqueness results were established by utilizing the reference ball technique in conjunction with the superposition of incident waves. With the aid of the reference ball technique, the a priori assumptions in \cite{XZZ18a} can be removed, see \cite{XZZ18b} for the details. Similar strategies of adding reference objects or sources to the scattering system for different models of phaseless inverse scattering problems can be found in \cite{DZG19, DLL18, JL18, JLZ18a, JLZ18b, ZhangDeyue20182}. For the numerical algorithms for the shape reconstruction from phaseless data, we refer to \cite{BLL2013, Bao2016, CH17, Ivanyshyn1, Ivanyshyn2, Ivanyshyn3, KR16, Kress, Lee2016, Li1, LLW17, Shin}.

In this paper, we consider the incident point sources and deal with the uniqueness issue concerning the inverse cavity scattering problems with phaseless total field data.
We rigorously prove in this paper that the location and shape of the obstacle as well as its boundary condition can be uniquely determined by the modulus of total fields at an admissible surface. To the best of our knowledge, this is a first uniqueness result in inverse cavity scattering problems with phaseless near-field data.  The main idea here is the utilization of the reference ball technique, superpositions of point sources, the reciprocity relations and the singularity of the total fields. We emphasize that the reference ball technique should be necessary for the phaseless inverse scattering problems for cavities owing to lack of the far-field pattern, and the reference ball can provide some information on the location of the cavity in devising effective numerical inversion schemes in comparison with the exterior inverse scattering problems (see, e.g. \cite{DZG19}), which will be our future work.

The rest of this paper is arranged as follows. In the next section, we present an introduction to the model problem. Section \ref{sec:obstacle} is devoted to the uniqueness results on phaseless inverse cavity scattering problem.


\section{Problem setting}\label{sec:problem_setup}

We begin this section with the precise formulations of the model cavity scattering problem. Assume $D \subset\mathbb{R}^3$ is an open and simply connected domain with
$C^2$ boundary $\partial D$.  Denote by $u^i$ the incident field. Then, the interior scattering problem for cavities can be formulated as: to find the scattered field $u^s$ which satisfies the following boundary value problem:
\begin{align}
\Delta u^s+ k^2 u^s= & 0 \quad \mathrm{in}\ D,\label{eq:Helmholtz} \\
\mathscr{B}u= & 0 \quad \mathrm{on}\ \partial D, \label{eq:boundary_condition}
\end{align}
where $u=u^i+u^s$ denotes the total field and $k>0$ is the wavenumber. Here $\mathscr{B}$ in \eqref{eq:boundary_condition} is the boundary operator defined by
\begin{equation}\label{BC}
\begin{cases}
\mathscr{B}u=u, & \text{for a sound-soft cavity},  \\
\mathscr{B}u=\dfrac{\partial u}{\partial \nu}+ \lambda u, & \text{for an impedance cavity},
\end{cases}
\end{equation}
where $\nu$ is the unit outward normal to $\partial D$, and $\lambda\in C(\partial D)$ is the impedance function satisfying $\Im(\lambda)\geq 0$. This boundary condition \eqref{BC} covers the Dirichlet/sound-soft boundary condition, the Neumann/sound-hard boundary condition ($\lambda=0$), and the impedance boundary condition ($\lambda\neq 0$). The existence of a solution to the direct scattering problem \eqref{eq:Helmholtz}--\eqref{eq:boundary_condition} is well known (see, e.g. \cite{Cakoni1, Colton3, Colton}).

Now, we turn to introducing the  interior inverse scattering problem for incident point sources with limited-aperture phaseless near-field data. To this end, we first introduce a reference ball $B$ as an extra artificial object to the scattering system such that $\overline{B}\subset\subset D$ with the impedance boundary condition
\begin{eqnarray}\label{Referenceball}
\displaystyle\frac{\partial u}{\partial \nu}+ i\lambda_0  u=0\quad \mathrm{on}\ \partial B,
\end{eqnarray}
where $\lambda_0$ is a positive constant, and the following definition of admissible surfaces.
\begin{definition}[Admissible surface]
An open surface $\Gamma$ is called an admissible surface with
respect to domain $\Omega$ if

\noindent (i) $\Omega\subset\mathbb{R}^3\backslash\overline{D}$ is bounded and simply-connected;

\noindent (ii) $\partial \Omega$ is analytic homeomorphic to $\mathbb{S}^2$;

\noindent (iii) $k^2$ is not a Dirichlet eigenvalue of $-\Delta$ in $\Omega$;

\noindent (iv) $\Gamma\subset\partial\Omega$ is a two-dimensional analytic manifold with nonvanishing measure.
\end{definition}

\begin{remark}
The artificial obstacle with impedance boundary condition \eqref{Referenceball} can also be founded in \cite{Qin} to remove the interior eigenvalues for the direct scattering problems and the reference ball technique has been used in \cite{Li, Colton1, Colton2} for the exterior inverse scattering problems.
\end{remark}
\begin{remark}
    We would like to point out that this requirement for the admissibility of $\Gamma$ is quite mild and thus can be easily fulfilled. For instance, $\Omega$ can be chosen as a ball whose radius is less than $\pi/k$ and $\Gamma$ is chosen as an arbitrary corresponding semisphere.
\end{remark}

For a generic point $z\in D\backslash\overline{B}$, the incident field $u^i$ due to the point source located at $z$ is given by
\begin{equation*}
u^i (x, z):=\frac{\mathrm{e}^{\mathrm{i} k|x-z|}}{4\pi |x-z|}, \quad x\in D\backslash(\overline{B}\cup\{z\}),
\end{equation*}
which is also known as the fundamental solution to the Helmholtz equation. Denote by $u^s(x,z)$  the near-field generated by $D$ and $B$ corresponding to the incident field $u^i(x, z)$. Let $u(x,z)=u^s(x,z)+u^i(x, z)$, $x\in D\backslash(\overline{B}\cup\{z\})$ be the total field.

For two generic and distinct source points $z_1, z_2\in D\backslash\overline{B}$, we denote by
\begin{equation}\label{incident}
 u^i(x; z_1,z_2):=u^i(x, z_1)+u^i(x, z_2),\quad x\in D\backslash(\overline{B}\cup\{z_1\}\cup\{z_2\}),
\end{equation}
the superposition of these point sources. Then, by the linearity of direct scattering problem, the near-field co-produced by $D$, $B$ and the incident wave $u^i(x; z_1,z_2)$ is given by
$$
u(x;z_1,z_2):=u(x,z_1)+u(x,z_2), \quad x\in D\backslash(\overline{B}\cup\{z_1\}\cup\{z_2\}).
$$

With these preparations,  we formulate the phaseless inverse scattering problems as the following.

\begin{problem}[Phaseless inverse scattering]\label{prob:obstacle}
Let $D$ be the impenetrable cavity with boundary condition $\mathscr{B}$. Assume that $\Gamma$ and $\Sigma$ are admissible surfaces with
respect to $\Omega$ and $G$, respectively, such that
$\overline{\Omega}\subset\subset G$ and $\overline {G}\subset\subset D\backslash\overline{B}$.  Given the phaseless
near-field data
 \begin{equation*}
 \begin{array}{ll}
 & \{|u(x,z_0)|: x\in \Sigma\}, \\
 & \{|u(x,z)|:  x\in \Sigma,\ z\in \Gamma\}, \\
 & \{|u(x,z_0)+u(x,z)|: x\in \Sigma,\ z\in \Gamma\}
 \end{array}
 \end{equation*}
 for a fixed wavenumber $k>0$ and a fixed  $z_0\in D\backslash(\overline{B}\cup\Gamma\cup\Sigma)$, determine the location and shape $\partial D$ as well as the boundary condition
 $\mathscr{B}$ for the cavity.
\end{problem}

We refer to Figure \ref{fig:illustration} for an illustration of the geometry setting of Problem \ref{prob:obstacle}. The uniqueness of this problem will be analyzed in the next section.

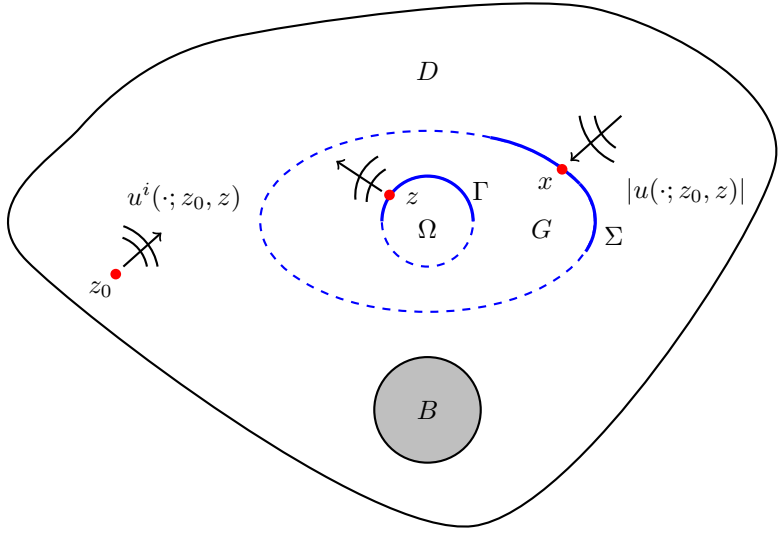
\begin{figure}
	\centering
	\newdimen\R 
	\R=0.5cm
	\begin{tikzpicture}[thick]
	\pgfmathsetseed{8}
	\draw plot [smooth cycle, samples=6, domain={1:8},  xshift=0cm, yshift=4cm] (\x*360/8+5*rnd:0.5cm+7cm*rnd) node at (1.5, 5.5) {$D$};
	
	\draw (1.5, 1) circle (0.7cm) [fill=lightgray] node at (1.5,1) {$B$};
	\draw [blue, dashed] (1.5, 3.5) ellipse [x radius=2.2cm, y radius=1.2cm] ;
	
	\draw [blue, dashed] (1.5, 3.5) circle (0.6cm); 
	\draw node at (1.5,3.4) {$\Omega$};
	\draw (2.1, 3.5) arc(0:180:0.6cm) [very thick, blue];
	\draw node at (2.2,3.9) {$\Gamma$};
	\draw node at (1.3,3.8) {$z$};
	\fill [red] (1.,3.85) circle (2pt);  
	\draw [->] (0.9, 3.9)--(0.3, 4.3);
	\draw (0.95,4.19) arc(120:180:0.5cm);
	\draw (0.85,4.36) arc(120:185:0.6cm);
	
	\fill [red] (-2.6,2.8) circle (2pt);  
	\draw (-2.8, 2.6) node   {$z_0$};
	\draw [->](-2.5,2.9)--(-2,3.35);
	\draw (-2.19, 2.9) arc(10:70:0.5cm);
	\draw (-2.05, 2.9) arc(10:70:0.7cm);
	\draw (-1.7, 3.5) node [above] {$u^i(\cdot; z_0, z)$};
	
	\draw [->](4.05, 4.9)-- (3.38,4.3);
	\draw (3.58, 3.1) arc(-35:45:0.7cm) [very thick, blue];
	\draw (3.5, 4) arc(45:80:2.2cm) [very thick, blue];
	\draw node at (3.05,4.) {$x$};
	\fill [red] (3.27,4.19) circle (2pt);  
	
	\draw node at (3.,3.4) {$G$};
	\draw node at (3.95,3.3) {$\Sigma$};
	\draw (3.7, 4.9) arc(188:238:0.6cm);
	\draw (3.5, 4.9) arc(188:245:0.8cm);
	
	\draw node at (4.9,3.9) {$|u(\cdot; z_0, z)|$}; 
	\end{tikzpicture}
	\caption{An illustration of the phaseless inverse scattering problem.} \label{fig:illustration}
\end{figure}

\section{Uniqueness for the phaseless inverse scattering}\label{sec:obstacle}

Now we present the uniqueness results on phaseless inverse scattering. The following theorem shows that Problem \ref{prob:obstacle} admits a unique solution, namely, the geometrical and physical information of the scatterer boundary can be simultaneously and uniquely determined from the modulus of near-fields.

\begin{theorem}\label{Thm1}
 Let $D_1$ and $D_2 $ be two impenetrable cavities with boundary conditions $\mathscr{B}_1$ and $\mathscr{B}_2$, respectively.  Assume that $\Gamma$ and $\Sigma$ are admissible surfaces with respect to $\Omega$ and $G$, respectively, such that
$\overline{\Omega}\subset\subset G$ and $\overline {G}\subset\subset (D_1\cap D_2)\backslash\overline{B}$. Suppose that the corresponding near-fields satisfy
that
   \begin{align}
   |u_1(x,z_0)|= & |u_2(x,z_0)|, \quad \forall x  \in \Sigma, \label{obstacle_condition1} \\
   |u_1(x,z)|= & |u_2(x,z)|, \quad \forall (x, z) \in \Sigma\times\Gamma \label{obstacle_condition2}
   \end{align}
   and
   \begin{eqnarray}\label{obstacle_condition3}
   |u_1(x,z_0)+u_1(x,z)|=|u_2(x,z_0)+u_2(x,z)|,\quad \forall (x, z) \in \Sigma\times\Gamma
   \end{eqnarray}
   for an arbitrarily fixed $z_0\in (D_1\cap D_2)  \backslash(\overline{B}\cup\Gamma\cup\Sigma)$. Then we have
 $D_1=D_2$ and $\mathscr{B}_1=\mathscr{B}_2$.
\end{theorem}
\begin{proof}
    From  \eqref{obstacle_condition1}, \eqref{obstacle_condition2} and \eqref{obstacle_condition3}, we have for all $x\in\Sigma, z\in\Gamma$
    \begin{equation}\label{Thm1equality1}
    \mathrm{Re}\left\{u_1(x,z_0) \overline{u_1(x,z)}\right\}
    =\mathrm{Re}\left\{u_2(x,z_0) \overline{u_2(x,z)}\right\},
    \end{equation}
    where the overline denotes the complex conjugate. According to \eqref{obstacle_condition1} and \eqref{obstacle_condition2}, we denote
    \begin{equation*}
        u_j(x,z_0)=r(x,z_0) \mathrm{e}^{\mathrm{i} \alpha_j(x,z_0)},\quad
        u_j(x,z)=s(x,z) \mathrm{e}^{\mathrm{i} \beta_j(x,z)},\quad j=1,2,
    \end{equation*}
    where $r(x,z_0)=|u_j(x,z_0)|$, $s(x,z)=|u_j(x,z)|$, $\alpha_j(x,z_0)$ and $\beta_j(x,z)$, are real-valued functions,  $j=1,2$.

    Since $\Sigma$ is an admissible surface of $G$, by definition 2.1 and the analyticity of $u_j(x,z)$ with respect to $x$, we have $s(x,z)\not\equiv 0$
    for $x \in \Sigma, z \in\Gamma$. Further, the continuity yields that there exists open sets $\tilde{\Sigma}\subset \Sigma$ and $\Gamma_0\subset\Gamma$
    such that $s(x,z)\neq 0$, $\forall (x,z)\in \tilde{\Sigma}\times\Gamma_0$. Similarly, we have $r(x,z_0)\not\equiv 0$ on $\tilde{\Sigma}$. Again,
    the continuity leads to $r(x,z_0)\neq 0$ on an open set $\Sigma_0 \subset \tilde{\Sigma}$. Therefore, we have $r(x,z_0)\neq 0$, $s(x,z)\neq 0,\ \forall (x, z) \in \Sigma_0\times\Gamma_0$. In addition, taking \eqref{Thm1equality1} into account, we derive that
    \begin{equation*}
        \cos[\alpha_1(x,z_0)-\beta_{1}(x,z)]=\cos[\alpha_2(x,z_0)-\beta_{2}(x,z)], \quad \forall (x, z) \in \Sigma_0\times\Gamma_0.
    \end{equation*}
    Hence, either
    \begin{eqnarray}\label{Thm1equality2}
    \alpha_1(x,z_0)-\alpha_2(x,z_0)=\beta_{1}(x,z)-\beta_{2}(x,z)+ 2m\pi, \quad \forall (x, z) \in \Sigma_0\times\Gamma_0
    \end{eqnarray}
    or
    \begin{eqnarray}\label{Thm1equality3}
    \alpha_1(x,z_0)+\alpha_2(x,z_0)=\beta_1(x,z)+\beta_{2}(x,z)+ 2m\pi, \quad \forall (x, z) \in \Sigma_0\times\Gamma_0
    \end{eqnarray}
    holds with some $m \in \mathbb{Z}$.

    First, we shall consider the case \eqref{Thm1equality2}. Since $z_0$ is fixed, let us define $\gamma(x):=\alpha_1(x,z_0)-\alpha_2(x,z_0)- 2m\pi$
    for $ x \in \Sigma_0$, and then, we deduce for all $ (x, z) \in \Sigma_0\times\Gamma_0$
    \begin{equation*}
        u_1(x,z)=s(x,z)\mathrm{e}^{\mathrm{i} \beta_{1}(x,z)}
        =s(x,z)\mathrm{e}^{\mathrm{i} \beta_{2}(x,z)+\mathrm{i} \gamma(x)}=u_2(x,z)\mathrm{e}^{\mathrm{i} \gamma(x)}.
    \end{equation*}
    From the reciprocity relation \cite[Theorem 2.1]{Qin} for point sources, we have
    \begin{equation*}
        u_1(z,x)= \mathrm{e}^{\mathrm{i} \gamma(x)}u_2(z,x), \quad \forall (x, z) \in \Sigma_0\times\Gamma_0.
    \end{equation*}
Then, for every $x\in \Sigma_0$, by using the analyticity of $u_j(z,x)$($j=1,2$) with respect to $z$, we have $u_1(z,x)= \mathrm{e}^{\mathrm{i} \gamma(x)}u_2(z,x),$   $\forall z\in \partial\Omega$.
Let $w(z,x)=u_1(z,x)-\mathrm{e}^{\mathrm{i} \gamma(x)}  u_2(z,x)$ and $D_0=D_1\cap D_2$, then
$$
\begin{cases}
\Delta w+ k^2 w=0 &\mathrm{in}\ \Omega, \\
w=0  &\mathrm{on}\ \partial \Omega.
\end{cases}
$$

By the assumption of $\Omega$ that $k^2$ is not a Dirichlet eigenvalue of $-\Delta$ in $\Omega$,  we find $w=0$ in $\Omega$. Now, the analyticity of
$u_j(z,x)(j=1,2)$ with respect to $z$ yields
    \begin{equation*}
        u_1(z,x)= \mathrm{e}^{\mathrm{i} \gamma(x)}u_2(z,x),\quad \forall z \in D_0\backslash (B\cup \{x\}).
    \end{equation*}
i.e., for all $ z \in D_0\backslash (B\cup \{x\})$,
\begin{equation}\label{eq:relation}
        u^s_1(z,x)+\Phi(z,x)=\mathrm{e}^{\mathrm{i} \gamma(x)} \left[ u^s_2(z,x)+\Phi(z,x)\right].
\end{equation}
By the Green's formula \cite[Theorem 2.5]{Colton}, one can readily deduce that $u^s_j(z,x)$ is bounded for $z\in D_j\backslash B$ $(j=1,2)$,
which, together with \eqref{eq:relation}, implies that $(1-\mathrm{e}^{\mathrm{i} \gamma(x)})\Phi(z,x)$ is bounded for $z \in D_0\backslash (B\cup \{x\})$. Hence, by letting $z\rightarrow x$, we obtain $\mathrm{e}^{\mathrm{i} \gamma(x)} = 1$, and again the reciprocity relation \cite[Theorem 2.1]{Qin} leads to
\begin{equation*}
    u_{1}^s(x,z) =  u_{2}^s(x,z),\quad   \forall (x, z) \in \Sigma_0\times \partial\Omega.
\end{equation*}
By a similar discussion of \eqref{eq:relation} for $u_1^s(x,z)- u_2^s(x,z)$ on $G$, we have
\begin{equation}\label{coincide}
    u_{1}^s(x,z) =  u_{2}^s(x,z),\quad   \forall (x, z) \in (D_0\backslash B)\times \partial\Omega.
\end{equation}

Next we are going to show that the case \eqref{Thm1equality3} does not hold. Suppose that \eqref{Thm1equality3} is true, then following a similar argument, we see that for every $x\in \Sigma_0$, there exists $\eta(x)$ such that $u_1(z,x)=\mathrm{e}^{\mathrm{i} \eta(x)} \overline{u_2(z,x)}$ for all $z \in D_0\backslash (B\cup \{x\})$, i.e.
$$
      u_1^s(z,x)+\Phi(z,x)=\mathrm{e}^{\mathrm{i} \eta(x)}  \overline{[u_2^s(z,x) +\Phi(z,x)]}.
$$
Then, from the boundedness of $u_{j}^s(z,x)$, it can be seen that $\Phi(z,x)-\mathrm{e}^{\mathrm{i} \eta(x)} \overline{\Phi(z,x)}$ is bounded for all $z \in D_0\backslash (B\cup \{x\})$. Since
$$
	\Phi(z,x)-\mathrm{e}^{\mathrm{i} \eta(x)} \overline{\Phi(z,x)}
	=  \left[1-\mathrm{e}^{\mathrm{i} \eta(x)}\right]\frac{\cos (k|z-x|)}{4\pi|z-x|}\\[2mm]
	+\mathrm{i}\left[1+\mathrm{e}^{\mathrm{i} \eta(x)}\right]\frac{\sin (k|z-x|) }{4\pi|z-x|},
$$
by letting $z\to x$, we deduce $\mathrm{e}^{\mathrm{i} \eta(x)}=1$, and thus, $u_1(z,x)= \overline{u_2(z,x)}$ for $z
 \in D_0\backslash (B\cup \{x\})$. Further, by using the  impedance  boundary condition  $\frac{\partial u_j(z,x)}{\partial \nu}+ \mathrm{i}\lambda_0  u_j(z,x)=0, z\in \partial B (j=1,2)$, we have
$$
\dfrac{\partial \overline{u_2(z,x)}}{\partial \nu}+ \mathrm{i}\lambda_0  \overline{u_2(z,x)}
=\dfrac{\partial u_1(z,x)}{\partial \nu}+ \mathrm{i}\lambda_0  u_1(z,x)=0, \quad z\in\ \partial B,
$$
which yields
$$
\dfrac{\partial u_2(z,x)}{\partial \nu}- \mathrm{i}\lambda_0  u_2(z,x)=0, \quad z\in\ \partial B.
$$
This is a contradiction. Hence, the case \eqref{Thm1equality3} does not hold.

Having verified \eqref{coincide}, we complete our proof as the consequences of two existing uniqueness results,  Theorem 2.1 in \cite{Qin1} and Theorem 3.1 in \cite{Qin2}.
\end{proof}

\begin{remark}
    We would like to point out that an analogous uniqueness result in two dimensions remains valid after appropriate modifications of the fundamental solution
    and the admissible surface. So we omit the 2D details.
\end{remark}




\section*{Acknowledgements}

D. Zhang and Y. Wang were supported by NSF of China under the grant 11671170. Y. Guo was supported by NSF of China under the grant 11601107, 41474102 and 11671111. The work of J.~Li was partially supported by the NSF of China under the grant No.~11571161 and the Shenzhen Sci-Tech Fund No.~JCYJ20170818153840322.




\end{document}